\documentclass[12pt]{amsart}
\usepackage[all]{xy}
\usepackage{amsmath}
\usepackage{amsfonts}
\usepackage{amssymb}
\usepackage{mathrsfs}

\DeclareFontEncoding{OT2}{}{} 
\newcommand{\textcyr}[1]{%
 {\fontencoding{OT2}\fontfamily{wncyr}\fontseries{m}\fontshape{n}
 \selectfont #1}}
\newcommand{\Sha}{{\!\!\lbe\mbox{\textcyr{Sh}}}}
\newcommand{\Bha}{{\!\be\lbe\mbox{\textcyr{B}}}}

\def\ra{\rightarrow}
\def \ra{\rightarrow}
\newtheorem{theorem}{Main Theorem\!\!}

\newtheorem{lemma}{Lemma}[section]
\newtheorem{corollary}[lemma]{Corollary}
\newtheorem{proposition}[lemma]{Proposition}
\theoremstyle{definition}

\theoremstyle{remark}
\newtheorem{remark}[lemma]{Remark}

\def\s{\mathscr }

\def\le{\kern 0.03em}

\def\krn{{\rm{Ker}}\e }

\def\e{\kern 0.08em}
\def\be{\kern -.1em}
\def\lbe{\kern -.025em}
\def\ra{\rightarrow}
\def\cok{{\rm{Coker}}\e}

\DeclareMathOperator{\Gal}{Gal}

\DeclareMathOperator{\Sel}{{\rm{Sel}}}

\DeclareMathOperator{\spec}{{\rm{Spec}}}

\def\bq{{\mathbb Q}}
\def\bz{{\mathbb Z}}

\begin{document}

\title[The Hasse principle over function fields]{On the Hasse principle for finite
group schemes over global function fields}
\author{Cristian D. Gonz\'{a}lez-Avil\'{e}s and Ki-Seng Tan}

\address{Departamento de Matem\'aticas, Universidad de La Serena, Chile}
\email{cgonzalez@userena.cl}

\address{Department of Mathematics\\
National Taiwan University\\
Taipei 10764, Taiwan}
\email{tan@math.ntu.edu.tw}


\thanks{C.G.-A. is partially supported by Fondecyt grant 1080025.
K.-S. T. is partially supported by the National Science Council of
Taiwan, NSC99-2115-M-002-002-MY3.}

\begin{abstract} Let $K$ be a global function field of
characteristic $p>0$ and let $M$ be a (commutative) finite and flat $K$-group scheme. We show that the kernel of the
canonical localization map $H^{1}\lbe(K,M)\longrightarrow
\prod_{\e\text{all}\e v}\be H^{1}\lbe(K_{v},M)$ in flat (fppf)
cohomology can be computed solely in terms of Galois cohomology. We then give applications to the case
where $M$ is the kernel of multiplication by $p^{\le m}$ on an
abelian variety defined over $K$.

\end{abstract}
\maketitle

\section{Statement of the main theorem}

Let $K$ be a global field and let $\overline{K}$ be a fixed
algebraic closure of $K$. Let $K^{\le\rm{s}}$ be the separable
closure of $K$ in $\overline{K}$ and set $G_{\lbe
K}=\Gal\le(K^{\le\rm{s}}\lbe/K)$. Further, for each prime $v$ of
$K$, let $\overline{K}_{\be v}$ be the completion of $\overline{K}$
at a fixed prime $\overline{v}$ of $\overline{K}$ lying above $v$
and let $K^{\le\rm{s}}_{v}$ denote the completion of $K^{\le\rm{s}}$
at the prime of $K^{\le\rm{s}}$ lying below $\overline{v}$. Set
$G_{v}=\Gal\le(K_{v}^{\le\rm{s}}\lbe/K_{v})$. If $M$ is a
commutative, finite and flat $K$-group scheme, let $H^{\le
i}(G_{\lbe K},M)$ (respectively, $H^{\le i}(G_{v},M)$) denote the
Galois cohomology group $H^{\le i}(G_{\lbe K},M(K^{\le\rm{s}}))$
(respectively, $H^{\le i}(G_{v},M(K_{v}^{\le\rm{s}}))$). The
validity of the Hasse principle for $M(K^{\le\rm{s}})$, i.e., the
injectivity of the canonical localization map $H^{\le 1}(G_{\lbe
K},M)\ra\prod_{\e\text{all}\e v} H^{\le 1}(G_{v},M)$ in Galois
cohomology, has been discussed in \cite{N,jan82}. See also
\cite{adt}, \S I.9, pp. 117-120. However, if $K$ is a global
function field of characteristic $p>0$ and $M$ has $p$-power order,
the injectivity of the canonical localization map $\beta^{\le
1}(K,M)\colon H^{\le 1}(K,M)\ra\prod_{\e\text{all}\e v} H^{\le
1}(K_{v},M)$ in flat (fppf) cohomology has not been discussed
before (but see \cite{mil70}, Lemma 1, for some particular
cases). In this paper we investigate this problem and show that the
injectivity of $\beta^{\le 1}(K,M)$ depends only on the finite
$G_{\lbe K}$-module $M(K^{\le\rm{s}})$, which may be regarded as the
maximal \'etale $K$-subgroup scheme of $M$. Indeed, let $\Sha^{\le
1}(K,M)=\krn\e\beta^{\le 1}(K,M)$ and set $\Sha^{1}(G_{\lbe
K},M)=\krn[\e H^{\le 1}(G_{\lbe K},M)\ra\prod_{\e\text{all}\e v}
H^{\le 1}(G_{v},M)\e]$. Then the following holds.

\begin{theorem}\label{t:etale} Let $K$ be a global function field of
characteristic $p>0$ and let $M$ be a commutative, finite and flat
$K$-group scheme. Let $v$ be any prime of $K$.
Then the inflation map $H^{\le 1}(G_{\lbe K},M)\ra H^{\le 1}(K,M)$
induces an isomorphism
$$
\krn\!\be\left[\e H^{\le 1}(G_{\lbe K},M)\ra H^{\le
1}(G_{v},M)\right]\simeq \krn\!\be\left[\e H^{\le 1}(K,M)\ra H^{\le
1}(K_{v},M)\right].
$$
In particular, $\Sha^{\le 1}(K,M)\simeq\Sha^{1}(G_{\lbe K},M)$.
\end{theorem}

Thus the Hasse principle holds for $M$, i.e., $\Sha^{\le 1}(K,M)=0$,
if, and only if, the Hasse principle holds for the $G_{\lbe
K}$-module $M(K^{\le\rm{s}})$. An interesting example is the
following. Let $A$ be an ordinary abelian variety over $K$ such that
the Kodaira-Spencer map has maximal rank. Then $A_{p^{\le
m}}\be(K^{\le\rm{s}})=0$ for every integer $m\geq 1$
\cite{vol95}, Proposition, p.1093, and we
conclude that the Hasse principle holds for $A_{p^{\le m}}$

The theorem is proved in the next Section. In Section 3, which
concludes the paper, we develop some applications.

\section{Proof of the main theorem}

We keep the notation introduced in the previous Section. In
addition, we will write $(\spec K)_{\rm{fl}}$ for the flat site on
$\spec K$ as defined in \cite{mil80}, II.1, p.47, and $H^{\le
r}(K,M)$ will denote $H^{r}((\spec K)_{\rm{fl}},M)$. We will see
presently that, in fact, $H^{\le r}(K,M)\simeq H^{r}((\spec
K)_{\rm{fppf}},M)$.

By a theorem of M.Raynaud (see \cite{ray79} or \cite{BBM}, Theorem
3.1.1, p. 110), there exist abelian varieties $A$ and $B$ defined
over $K$ and an exact sequence of $K$-group schemes
\begin{equation}\label{seq}
0\ra M\overset{\iota}\ra A\overset{\psi}\ra B\ra 0,
\end{equation}
where $\iota$ is a closed immersion. For $r\geq 0$, let
$\iota^{(r)}\colon H^{\le r}(K,M)\ra H^{\le r}(K,A)$ and
$\psi^{(r)}\colon H^{\le r}(K,A)\ra H^{\le r}(K,B)$ be the maps
induced by $\iota$ and $\psi$. The long exact flat cohomology
sequence associated to \eqref{seq} yields an exact sequence
$$
0\ra \cok\psi^{(r-1)}\ra H^{\le r}(K,M)\ra \krn\psi^{(r)}\ra 0,
$$
where $r\geq 1$. Since the groups $H^{\le r}(K,A)$ and $H^{\le
r}(K,B)$ coincide with the corresponding \'etale and fppf cohomology
groups \cite{mil80}, Theorem III.3.9, p.114, and \cite{grb68},
Theorem 11.7, p.180,  we conclude that $H^{\le r}(K,M)=H^{r}((\spec
K)_{\rm{fppf}},M)$.

\begin{lemma}\label{l:insp} Let $v$ be any prime of $K$. If $A$
is an abelian variety defined over $K$, then
$A(K^{\le\rm{s}})=A\big(\e\overline {K}\e\big)\cap A(K_{\lbe
v}^{\le\rm{s}})$, where the intersection takes place inside
$A\big(\e\overline {K}_{\be v}\big)$.
\end{lemma}
\begin{proof} Let $F/K$ be a finite subextension of
$K^{\le\rm{s}}/K$ and let $F_{v}$ denote the completion of $F$ at
the prime of $F$ lying below $\overline{v}$. Choose an element $t\in
F$ such that $F_{v}=k((t))$, where $k$ is the field of constants of
$F$, and let $m\geq 1$ be an integer. Then $F_{v}^{\e
p^{-m}}=k((t^{\e p^{-m}}))=\sum_{i=1}^{p^{m}}(t^{\e
p^{-m}})^{i}F_{v}\subset F^{\e p^{-m}}F_{v}$, whence $F_{v}^{\e
p^{-m}}=F^{\e p^{-m}}F_{v}$. Now let $a\in {\overline{K}}$ be
inseparable over $K$. Then there exists an integer $m\geq 1$ and an
extension $F/K$ as above such that $K(a)=F^{\e p^{-m}}$.
Consequently $a\in K(a)\cdot F_{v} =F^{\e p^{-m}}F_{v}=F_{v}^{\e
p^{-m}}$, whence $a$ is also inseparable over $K_{v}$. This shows
that $K^{\le\rm{s}}=\overline {K}\cap K_{v}^{\le\rm{s}}$. Now let
$V\subset\mathbb A_{K}^{n}$ be an affine $K$-variety and let
$P=(x_{1},...,x_{n})\in V\big(\e\overline {K}\,\big)\cap
V(K_{v}^{\le\rm{s}})$. Then each $x_i\in \overline {K}\cap
K_{v}^{\le\rm{s}}=K^{\le\rm{s}}$, whence $P\in V(K^{\le\rm{s}})$.
Thus $V(K^{\le\rm{s}})=V\big(\e\overline {K}\,\big)\cap
V(K_{v}^{\le\rm{s}})$, and the lemma is now clear since $A$ is
covered by affine $K$-varieties.
\end{proof}

If $v$ is a prime of $K$, we will write $\psi_{v}=\psi\times_{\spec
K}\spec K_{v}$. Since $H^{\le 1}(K^{\le\rm{s}},A)=H^{\le
1}(K^{\le\rm{s}}_{v},A)=0$, the exact sequence \eqref{seq} yields a
commutative diagram
\begin{equation}\label{diag}
\xymatrix{B(K^{\le\rm{s}})/\psi(A(K^{\le\rm{s}}))\ar[d]
\ar[r]^(.58){\sim}&H^{\le 1}(K^{\le\rm{s}},M)
\ar[d]\\
B(K_{v}^{\le\rm{s}})/\psi_{v}(A(K_{v}^{\le\rm{s}}))\ar[r]^(.58){\sim}
&H^{\le 1}(K^{\le\rm{s}}_{v},M).}
\end{equation}

\begin{lemma}\label{l:main}
Let $v$ be a prime of $K$. Then the canonical map
$$
B(K^{\le\rm{s}})/\psi(A(K^{\le\rm{s}}))\ra B(K_{v}^{\le\rm{s}})/
\psi_{v}(A(K_{v}^{\le\rm{s}}))
$$
is injective.
\end{lemma}
\begin{proof} Write $M=\spec R$, where $R$ is a finite $K$-algebra, and identify
$M\big(\e\overline {K}_{\be v}\big)$ with $\text{Hom}_{K_{\lbe
v}}\be(K_{\lbe v}\otimes_{K} R, \overline{K}_{\be v})$. If $s\in
M\big(\e\overline {K}_{\lbe v}\big)$, then the image of the
composition $\tilde{f}\colon R\ra K_{\lbe v}\otimes_{K} R
\overset{s}\ra\overline{K}_{\be v}$ is a finite $K$-algebra and so,
in fact, a finite field extension of $K$. Consequently, $\tilde{f}$
factors through some $f\in \text{Hom}_{K}\big(R,
\overline{K}\e\big)=M\big(\e\overline {K}\e\big)$. This implies that
$M\big(\e\overline {K}_{v}\big)= M\big(\e\overline {K}\big)$. Now
let $P\in B(K^{\le\rm{s}})\cap
\psi_{v}(A(K_{v}^{\le\rm{s}}))\subseteq B(K_{v}^{\le\rm{s}})$ and
let $Q\in A(K_{v}^{\le\rm{s}})$ be such that $P=\psi_{v}(Q)$. Since
$A\big(\e\overline {K}\e\big)\overset{\psi}\ra B\big(\e\overline
{K}\e\big) $ is surjective, there exists an $R\in A\big(\e\overline
{K}\e\big)$ such that $\psi (R)=P$. Then $R-Q\in M\big(\e\overline
{K}_v\e\big)= M\big(\e\overline {K}\e\big)$. This shows that $Q\in
A\big(\e\overline {K}\e\big)\cap
A(K_{v}^{\le\rm{s}})=A(K^{\le\rm{s}})$, by the previous lemma. Thus
$P=\psi(Q)\in \psi(A(K^{\le\rm{s}}))$, as desired.

\end{proof}

The above lemma and diagram \eqref{diag} show that the localization
map $H^{\le 1}(K^{\le\rm{s}},M)\ra H^1(K^{\le\rm{s}}_{\lbe v},M)$ is
injective. The main theorem is now immediate from the exact
commutative diagram
$$
\xymatrix{ 0\ar[r] & H^{\le 1}(G_{\lbe K},M)\ar[r]\ar[d] &
H^{1}(K,M)\ar[r]\ar[d] & H^{\le 1}(K^{\le\rm{s}},M)
\ar@{^{(}->}[d]\\
0\ar[r]& H^{\le 1}(G_{v},M)\ar[r]& H^{1}(K_{v},M)\ar[r] & H^{\le
1}(K^{\le\rm{s}}_{v},M),}
$$
whose rows are the inflation-restriction exact sequences in flat
cohomology \cite{S}, p.422, line -12.

\section{Applications}

Let $K$ and $M$ be as in the previous Section. We will write $K(M)$
for the subfield of $K^{\le\rm{s}}$ fixed by $\krn\!\be\left[\e
G_{\lbe K}\ra {\rm{Aut}}\big(M(K^{\le\rm{s}})\big)\right]$. We note
that the Hasse principle is known to hold for $M(K^{\le\rm{s}})$
under any of the following hypotheses:
\begin{enumerate}
\item[(a)] $\Gal\le(K(M)\lbe/K)\subseteq {\rm{Aut}}\big(M(K^{\le\rm{s}})\big)$
is cyclic. See \cite{adt},
Lemma I.9.3, p.118.
\item[(b)] $M(K^{\le\rm{s}})$ is a simple $G_{\lbe K}$-module such that
$p\e M(K^{\le\rm{s}})=0$ and $\Gal\le(K(M)\lbe/K)$ is a $p$-solvable
group, i.e., $\Gal\le(K(M)\lbe/K)$ has a composition series whose
factors of order divisible by $p$ are cyclic. See \cite{adt},
Theorem I.9.2(a), p.117.
\item[(c)] There exists a set $T$ of primes of $K$, containing the set $S$ of all
primes of $K$ which split completely in $K(M)$, such that
$T\setminus S$ has Dirichlet density zero and
$[K(M):K]={\rm{l.c.m.}}\{[K(M)_{v}\!:\! K_{v}] \colon v\in T\}$. See
\cite{nsw}, Theorem 9.1.9(iii), p.528.
\end{enumerate}

In this Section we focus on case (a) above when $M=A_{p^{\le m}}$ is
the $p^{\le m}$-torsion subgroup scheme of an abelian variety $A$
defined over $K$. More precisely, we are interested in the class of
abelian varieties $A$ such that $A_{p^{\le m}}\be (K^{\le\rm{s}})$
is cyclic, for then $\Gal\le(K(A_{p^{\le m}})\lbe/K)\hookrightarrow
{\rm{Aut}}\big(A_{p^{\le m}} ( K^{\le\rm{s}})\big)$ is cyclic as
well if $p$ is odd or $m\leq 2$ and (a) applies. Clearly, this class
contains all ordinary abelian varieties $A$ such that the associated
Kodaira-Spencer map has maximal rank since, as noted in Section 1,
$A_{p^{\le m}}\be (K^{\le\rm{s}})$ is in fact zero. To find more
examples, recall that $A_{p^{\le
m}}\be\big(\e\overline{K}\e\big)\simeq (\bz/p^{\le m}\bz)^f$ for
some integer $f$ (called the {\it $p$-rank of $A$}) such that $0\leq
f\leq \dim A$. Thus, if $f\leq 1$, then $A_{p^{\le m}}\be
(K^{\le\rm{s}})$ is cyclic. Clearly, the condition $f\leq 1$ holds
if $A$ is an elliptic curve, but there exist higher-dimensional
abelian varieties $A$ having $f\leq 1$. See \cite{P}, \S 4.

\begin{remark} Clearly, $\Gal\le(K(A_{p^{\le m}})\lbe/K)$ may be
cyclic even if $A_{p^{\le m}}\be (K^{\le\rm{s}})$ is not. For
example, let $k$ be the (finite) field of constants of $K$, let
$A_{0}$ be an abelian variety defined over $k$ and let
$A=A_{0}\times_{\spec k}\spec K$ be the constant abelian variety
over $K$ defined by $A_{0}$. Then
$A(K^{\le\rm{s}})_{\rm{tors}}=A_{0}\big(\e\overline{k}\e\big)$, and
it follows that $\Gal\le(K(A_{p^{\le
m}})\lbe/K)\simeq\Gal(k^{\e\prime}/k)$ for some finite extension
$k^{\e\prime}$ of $k$. Consequently $\Gal\le(K(A_{p^{\le
m}})\lbe/K)$ is cyclic and the Hasse principle holds for $A_{p^{\le
m}}(K^{\le\rm{s}})$.
\end{remark}

We will write $A\{p\}$ for the $p$-divisible group attached to $A$,
i.e., $A\{p\}=\varinjlim_{m} A_{p^{\le m}}$. If $B$ is an abelian
group, $B(p)=\bigcup_{m}\be B_{\e p^{\le m}}$ is the $p$-primary
component of its torsion subgroup, $B^{\wedge}=\varprojlim_{\,m}\be
B/p^{\le m}$ is the adic completion of $B$ and $T_{p}\e
B=\varprojlim_{m}\be B_{\e p^{\le m}}$ is the $p$-adic Tate module
of $B$. Further, if $B$ is a topological abelian group, $B^{D}$ will
denote $\text{Hom}_{\e\text{cont.}}(B,\bq/\e\bz)$ endowed with the
compact-open topology, where $\bq/\e\bz$ is given the discrete
topology.

Let $X$ denote the unique smooth, projective and irreducible curve
over the field of constants of $K$ having function field $K$. If $A$ is an abelian variety over $K$, we will write $\s A$ for the N\'{e}ron model of $A$ over $X$.

The following statement is immediate from the main theorem and the
above remarks.
\begin{proposition}\label{p:hasse}
Let $A$ be an abelian variety defined over $K$ and let $m$ be a
positive integer. Assume that $A_{p^{\le m}}\be( K^{\le\rm{s}})$ is
cyclic. Assume, in addition, that $m\leq 2$ if $p=2$. Then the
localization map in flat cohomology
$$
H^{\le 1}(K,A_{p^{\le m}}\be)\ra\prod_{{\rm{all}}\;v} H^{\e
1}(K_{v},A_{p^{\le m}}\be)
$$
is injective.\qed
\end{proposition}

The next lemma confirms a long-standing and widely-held expectation.

\begin{lemma} $H^{\e 2}(K,A)=0$.
\end{lemma}
\begin{proof} Since $H^{\le 2}(K_{v},A\be)=0$ for every $v$ \cite{adt}, Theorem III.7.8, p.285, it suffices to check that
$\Sha^{\le 2}(A)=0$. For any integer $n$, there exists a canonical exact sequence of flat cohomology groups
$$
0\ra H^{\le 1}(K,A\be)/n\ra H^{\le 2}(K,A_{n}\be)\ra
H^{\le 2}(K,A\be)_{n}\ra 0.
$$
Since the Galois cohomology groups $H^{\le i}(K,A\be)$ are torsion in degrees $i\geq 1$ and $\bq/\bz$ is divisible, the direct limit over $n$ of the above exact sequences yields a canonical isomorphism
$H^{\le 2}(K,A\be)=\varinjlim_{\e n}H^{\le 2}(K,A_{n}\be)$. An analogous isomorphism exists over $K_{v}$ for each prime $v$ of $K$, and we conclude that $\Sha^{2}(A)$ is canonically isomorphic to $\varinjlim_{\e n}\Sha^{2}(A_{n})$. Now, by Poitou-Tate duality \cite{adt}, Theorem I.4.10(a), p.57, and \cite{ga09}, Theorem 4.8, the latter group is canonically isomorphic to the Pontryagin dual of $\varprojlim_{n}\Sha^{1}(A^{t}_{n})$, where $A^{t}$ is the dual abelian variety of $A$. Thus, it suffices to show that $\varprojlim_{n}\Sha^{1}(A^{t}_{n})=0$. Let $U$ be the largest open subscheme of $X$ such that $A^{t}_{n}$ extends to a finite and flat $U$-group scheme $\s A^{\le t}_{n}$. For each closed point $v$ of $U$, let $\mathcal O_{v}$ denote the completion of the local ring of $U$ at $v$. Now let $V$ be any nonempty open subscheme of $U$. By the computations at the beginning of \cite{adt}, III.7, p.280, and the localization sequence \cite{mil80}, Proposition III.1.25, p.92, there exists an exact sequence
$$
0\ra H^{\le 1}(U,\s A^{\le t}_{n})\ra H^{\le 1}(V,\s A^{\le t}_{n})\ra\bigoplus_{v\in U\e\setminus V}H^{\le 1}(K_{v}, A^{t}_{n})/H^{\le 1}(\mathcal O_{v}, \s A^{\le t}_{n}).
$$
Taking the direct limit over $V$ in the above sequence and using
\cite{ga09}, Lemma 2.3, we obtain an exact sequence
$$
0\ra H^{\le 1}(U,\s A^{\le t}_{n})\ra H^{\le 1}(K,A^{t}_{n})\ra\prod_{v\in U}H^{\le 1}(K_{v}, A^{t}_{n})/H^{\le 1}(\mathcal O_{v}, \s A^{\le t}_{n}),
$$
where the product extends over all closed points of $U$. The exactness of the last sequence shows the injectivity of the right-hand vertical map in the diagram below
$$
\xymatrix{ 0\ar[r] & H^{\le 1}(U,\s A^{\le t}_{n})\ar[r]\ar[d] &
H^{\le 1}(K,A^{t}_{n})\ar[r]\ar[d] & H^{\le 1}(K,A^{t}_{n})/H^{\le 1}(U,\s A^{\le t}_{n})
\ar@{^{(}->}[d]\\
0\ar[r]& \displaystyle\prod_{v\in U}H^{\le 1}(\mathcal O_{v}, \s A^{\le t}_{n})\ar[r]& \displaystyle\prod_{v\in U}H^{\le 1}(K_{v}, A^{t}_{n})\ar[r] & \displaystyle\prod_{v\in U}H^{\le 1}(K_{v}, A^{t}_{n})/H^{\le 1}(\mathcal O_{v},\s A^{\le t}_{n}).}
$$
We conclude that $\Sha^{1}_{U}(A^{t}_{n}):=\krn[\e H^{\le 1}(K,A^{t}_{n})\ra\prod_{\e v\in U}H^{\le 1}(K_{v}, A^{t}_{n})]$ equals
$$
\overline{H}^{\e 1}(U,\s A^{\le t}_{n}):=\krn\!\left[\e H^{\le 1}(U,\s A^{\le t}_{n})\ra\prod_{v\in U}H^{\le 1}(\mathcal O_{v},\s A^{\le t}_{n})\right].
$$
Now, it is shown in \cite{mil72}, Propositions 5 and 6, that
$\varprojlim_{n}\overline{H}^{\e 1}(U,\s A^{\le t}_{n})=0$, whence $\varprojlim_{n}\Sha^{1}_{U}(A^{t}_{n})=0$. Now the exact sequence $$
0\ra \Sha^{1}(A^{t}_{n})\ra \Sha^{1}_{U}(A^{t}_{n})\ra\displaystyle\prod_{v\notin U}H^{\le 1}(K_{v}, A^{t}_{n})
$$
shows that $\varprojlim_{n}\Sha^{1}(A^{t}_{n})=0$, as desired.
\end{proof}

\begin{remark} With the notation of the above proof, the kernel-cokernel exact sequence \cite{adt}, Proposition I.0.24, p.16, for the pair of maps
$$
H^{\le 1}(K,A^{t}_{n})\ra\displaystyle
\bigoplus_{\text{all $v$}}H^{\le 1}(K_{v},A^{t}_{n})\ra
\displaystyle\bigoplus_{\text{all $v$}}H^{\le 1}(K_{v},A^{t})
$$
yields an exact sequence
$$
0\ra \Sha^{1}(A^{t}_{n})\ra\text{Sel}(A^{t})_{n}\ra\displaystyle
\bigoplus_{\text{all $v$}}H^{0}(K_{v},A^{t})/n,
$$
where $\text{Sel}(A^{t})_{n}:=\krn[H^{\le 1}(K,A^{t}_{n})\ra
\bigoplus_{\text{all $v$}}H^{\le 1}(K_{v},A^{t})]$. Since $\varprojlim_{n}\Sha^{1}(A^{t}_{n})=0$ as shown above, the inverse limit over $n$ of the preceding sequences yields an injection
$$
\varprojlim_{n}\text{Sel}(A^{t})_{n}\hookrightarrow
\prod_{\text{all $v$}}\varprojlim_{n}H^{0}(K_{v},A^{t})/n.
$$
This injectivity was claimed in \cite{gat07}, p.300, line -8, but the ``proof\e" given there is inadequate and should be replaced by the above one.
\end{remark}

\begin{proposition}\label{c:h2}
Let $A$ be an abelian variety over $K$ and let $A^{t}$ be the
corresponding dual abelian variety. Assume that $A^{t}_{p^{\le
m}}\be(K^{\le\rm{s}})$ is cyclic, where $m$ is a positive integer
such that $m\leq 2$ if $p=2$. Then the localization maps
$$
H^{\le 2}(K,A_{p^{\le m}}\be)\ra\bigoplus_{{\rm{all}}\;v} H^{\le
2}(K_{v},A_{p^{\le m}}\be)
$$
and
$$
H^{\le 1}(K,A\be)/p^{\le m}\ra\bigoplus_{{\rm{all}}\;v} H^{\le
1}(K_{v},A\be)/p^{\le m}
$$
are injective.
\end{proposition}
\begin{proof} The injectivity of the first map is immediate from Proposition 3.2 and
Poitou-Tate duality \cite{ga09}, Theorem 4.8. On the other hand, the
lemma and the long exact flat cohomology sequence associated to
$0\ra A_{p^{\le m}}\ra A\overset{p^{\le m}}\ra A\ra 0$ over $K$ and
over $K_{v}$ for each $v$ identifies the second map of the statement
with the first, thereby completing the proof.
\end{proof}

\begin{remark} When $A$ is an elliptic curve over a number field and $p$ is any prime,
the injectivity of the second map in the above proposition was first
established in \cite{cas62}, Lemma 6.1, p.107.
\end{remark}

Let $\prod_{\e\text{all}\;
v}^{\e\prime} H^1(K_v,A_{p^{\le m}}\!)$ denote the restricted
product of the groups $H^{\le 1}(K_{v},A_{p^{\le m}}\!)$ with
respect to the subgroups $H^{\le 1}(\mathcal O_{v}, \s A_{p^{\le
m}}\be)$.

\begin{proposition}\label{p:tate}
Let $A$ be an abelian variety over $K$ and let $m$ be a positive
integer such that $m\leq 2$ if $p=2$. Assume that both $A_{p^{\le
m}}\be(K^{\le\rm{s}})$ and $A^{t}_{p^{\le m}}\be(K^{\le\rm{s}})$ are
cyclic. Then there exist exact sequences
$$
0\ra A_{p^{\le m}}\be(K)\ra \prod_{{\rm{all}}\; v} A_{p^{\le
m}}\be(K_v) \ra H^{\le 2}(K, A^{t}_{p^{\le m}}\be)^D\ra 0,
$$
$$
0\ra H^{\le 1}(K, A_{p^{\le m}}\be)\ra\underset{{\rm{all}}\;
v}\prod^{\e\prime} H^{\le 1}(K_v,A_{p^{\le m}}\be)\ra H^{\le 1}(K,
A^t_{p^{\le m}}\be)^D\ra 0
$$
and
$$
0\ra H^{\le 2}(K, A_{p^{\le m}}\be)\ra\bigoplus_{{\rm{all}}\; v}
H^{\le 2}(K_{v},A_{p^{\le m}}\be)\ra A^{t}_{p^{\le m}}\be(K)\ra 0.
$$
\end{proposition}
\begin{proof} This is immediate from Propositions 3.2 and 3.5 and the Poitou-Tate
exact sequence in flat-cohomology \cite{ga09}, Theorem 4.11.
\end{proof}

Finally, set
$$
\Sel(A)_{p^{\le m}}=\krn\!\left[H^{1}(K,A_{p^{\le m}})\ra
\bigoplus_{\text{all $v$}}H^{1}(K_{v},A)\right]
$$
and define $T_{p}\Sel(A)=\varprojlim_{\e m}\Sel(A)_{p^{\le m}}$.
Further, recall the group
$$
\Sha^{1}(A)=\krn\!\left[H^{1}(K,A)\ra\bigoplus_{\text{all
$v$}}H^{1}(K_{v},A)\right].
$$
Now define $H^{\le 1}\big(K, T_{p}A^{t}\e\big)=\varprojlim_{\e
m}H^{\le 1}(K, A^t_{p^{\le m}}\be)$.

\begin{corollary} Under the hypotheses of the proposition, there
exist canonical exact sequences
$$
0\ra \Sha^{1}(A)(p)\ra\frac{\prod_{\e {\rm{all}}\,
v}A(K_{v})\otimes\bq_{\le p}/\bz_{\e p}}{A(K)\otimes\bq_{\le
p}/\bz_{\e p}}\ra H^{\le 1}(K,
T_{p}A^{t})^{D}\ra(T_{p}\Sel(A^{t}))^D\ra 0,
$$
$$
0\ra T_{p}\Sha^{1}(A)\ra\left(\e\textstyle\prod_{\e {\rm{all}}\,
v}A(K_{v})^{\wedge}\,\right)/A(K)^{\wedge}\ra H^{\le 1}(K,
A^t\{p\}\be)^{D}
$$
and
$$
0\ra T_{p}\Sel(A)\ra\prod_{\e {\rm{all}}\, v}A(K_{v})^{\wedge}\ra
H^{\le 1}(K, A^t\{p\}\be)^{D}.
$$
\end{corollary}
\begin{proof} Let $m\geq 1$ be an integer. The exact
commutative diagram
$$
\xymatrix{ 0\ar[r] & A(K)/p^{\le m}\ar[r]\ar@{^{(}->}[d] &
H^{1}(K,A_{p^{\le
m}})\ar[r]\ar@{^{(}->}[d] & H^{\le 1}(K,A)_{p^{\le m}}\ar[d]\ar[r]&0\\
0\ar[r] & \displaystyle\prod_{{\rm{all}}\, v}A(K_{v})/p^{\le
m}\ar[r]& \displaystyle{\prod_{{\rm{all}}\,
v}}^{\,\prime}H^{1}(K_{v},A_{p^{\le m}})\ar[r]&
\displaystyle\bigoplus_{{\rm{all}}\, v}H^{\le 1}(K_{v},A)_{p^{\le
m}}\ar[r]&0,}
$$
yields an exact sequence of profinite abelian groups
\begin{equation}\label{seq2}
0\ra \Sha^{1}(\be A)_{p^{\le m}}\ra\frac{\prod_{\e {\rm{all}}\,
v}A(K_{v})/p^{\le m}}{A(K)/p^{\le m}}\ra H^{\le 1}(K, A^t_{p^{\le
m}}\be)^D\ra\Bha_{m}\be(A)\ra 0,
\end{equation}
where $\Bha_{m}(A)=\cok\left[H^{\le 1}(K,A)_{p^{\le
m}}\ra\bigoplus_{{\rm{all}}\, v}H^{\le 1}(K_{v},A)_{p^{\le
m}}\right]$. By the main theorem of \cite{gat07},
$\varinjlim_{m}\be\Bha_{m}(A)\simeq(T_{p}\Sel(A^{t}))^D$ and the
first exact sequence of the statement follows by taking the direct
limit over $m$ in \eqref{seq2}. On the other hand, since the inverse
limit functor is exact on the category of profinite groups
\cite{RZ}, Proposition 2.2.4, p.32, the inverse limit over $m$ of
the sequences \eqref{seq2} is the second exact sequence of the
statement. The third exact sequence follows from the second and the
exact commutative diagram
$$
\xymatrix{ 0\ar[r] & A(K)^{\wedge}\ar[r]\ar@{=}[d] &
T_{p}\Sel(A)\ar[r]\ar@{^{(}->}[d] & T_{p}\Sha^{1}(A)\ar@{^{(}->}[d]\ar[r]&0\\
0\ar[r] & A(K)^{\wedge}\ar[r]& \displaystyle\prod_{{\rm{all}}\,
v}A(K_{v})^{\wedge}\ar[r]& \left(\e\prod_{\e {\rm{all}}\,
v}A(K_{v})^{\wedge}\,\right)/A(K)^{\wedge}\ar[r]&0.}
$$
\end{proof}

\section{Concluding remarks}
Let $K$ and $M$ be as in Section 1 and assume that $M$ has $p$-power order. Further, let $M^{*}$ be the Cartier dual of $M$. Since $\Sha^{1}(K,M^{*})\simeq\Sha^{1}(G_{\lbe K},M^{*})$ and there exists a perfect pairing of finite groups
\begin{equation}\label{pair}
\Sha^{1}(K,M^{*})\times \Sha^{\le 2}(K,M)\ra \bq/\bz
\end{equation}
\cite{ga09}, Theorem 4.8, it is natural to expect a Galois-cohomological description of $\Sha^{\le 2}(K,M)$.
Note, however, that the natural guess $\Sha^{\le 2}(K,M)\simeq\Sha^{\le 2}(G_{\lbe K},M)$ is incorrect, since the latter group is zero
(because the $p$-cohomological dimension of $G_{\lbe K}$ is $\leq 1$). To obtain the correct answer, we proceed as follows.
Since  $H^{\le i}(K^{\le\rm{s}},A)=H^{\le i}(K^{\le\rm{s}},B)=0$ for all $i\geq 1$, the exact sequence \eqref{seq}
shows that $H^{\le i}(K^{\le\rm{s}},M)=0$ for all $i\geq 2$. On the other hand,
$H^{\le i}(G_{\lbe K},M)=0$ for all $i\geq 2$ as well, since $\text{cd}_{p}(G_{\lbe K})\leq 1$. Now the exact sequence of terms of low
degree belonging to the Hochschild-Serre spectral sequence $H^{\le i}(G_{\lbe K}, H^{\le j}(K^{\le\rm{s}},M) )\Rightarrow
H^{\le i+j}(K,M)$ yields a canonical isomorphism
$$H^{\le 2}(K,M)\simeq H^{\le 1}(G_{\lbe K}, H^{\le 1}(K^{\le\rm{s}},M) )\simeq H^{\le 1}(G_{\lbe K}, B(K^{\le\rm{s}} )/\psi(A(K^{\le\rm{s}} ) )).$$
Analogous isomorphisms exist over $K_v$ for every prime $v$ of $K$, and we conclude that
$$
\Sha^{\le 2}(K,M)\simeq \Sha^{\le 1}(G_K,B/\psi(A)).
$$
For example, if $M=A_{p^{\le m}}$ for some abelian variety $A$ over $K$, then $\Sha^{\le 2}(K,A_{p^{\le m}}\lbe)\simeq \Sha^{\le 1}(G_K,A/p^{\le m})$ and the pairing \eqref{pair} takes the form
$$
\Sha^{\le 1}(G_K,A^{t}_{p^{\le m}})\times \Sha^{\le 1}(G_K,A/{p^m})\ra\bq/\bz.
$$
In particular, if $A^{t}_{p^{\le m}}(K^{\le\rm{s}})$ is cyclic with $m\leq 2$ if $p=2$, then Proposition 3.2 applied to $A^{t}$ and the perfectness of the above pairing yield $\Sha^{\le 1}(G_K,A/{p^{\le m}})=0$, i.e., the Hasse principle holds for the $G_{\lbe K}$-module $A(K^{\le\rm{s}})/p^{\le m}$.

\end{document}